\documentclass{amsart}
\usepackage{amsmath}
\usepackage{amsfonts}
\usepackage{amssymb}
\usepackage{graphicx}
\usepackage{float, verbatim}
\usepackage{enumerate}
\usepackage{color}
\usepackage{transparent}
\usepackage{hyperref}
\usepackage{tikz}
\usepackage{xcolor}
\usepackage{pgfplots}
\usepackage{caption}
\usepackage{subcaption}

\newcommand{\Assouad}{\dim_{\mathrm{A}}}

\newcommand{\Lower}{\dim_{\mathrm{L}}}
\newcommand{\Haus}{\dim_{\mathrm{H}}}
\newcommand{\Boxd}{\dim_{\mathrm{B}}}
\newcommand{\uBoxd}{\overline{\dim}_{\mathrm{B}}}
\newcommand{\lBoxd}{\underline{\dim}_{\mathrm{B}}}

\newcommand{\Ld}{\dim_{\mathrm{L}}\,}

\newcommand{\supp}{\mathrm{supp}}
\newtheorem*{thm*}{Theorem}
\newtheorem*{conj*}{Conjecture}

\newtheorem{thm}{Theorem}[section]

\newtheorem{lma}[thm]{Lemma}
\newtheorem{cor}[thm]{Corollary}
\newtheorem{defn}[thm]{Definition}

\newtheorem{prop}[thm]{Proposition}
\newtheorem{conj}[thm]{Conjecture}
\newtheorem{rem}[thm]{Remark}

\begin{document}

\title{Dimension growth for iterated sumsets}

\author{Jonathan M. Fraser}
\address{Jonathan M. Fraser\\
School of Mathematics \& Statistics\\University of St Andrews\\ St Andrews\\ KY16 9SS\\ UK }
\curraddr{}
\email{jmf32@st-andrews.ac.uk}
\thanks{JMF was financially supported by a \emph{Leverhulme Trust Research Fellowship} (RF-2016-500) and an \emph{ EPSRC Standard Grant} (EP/R015104/1).}

\author{Douglas  C. Howroyd}
\address{Douglas C. Howroyd\\
	School of Mathematics \& Statistics\\University of St Andrews\\ St Andrews\\ KY16 9SS\\ UK }
\curraddr{}
\email{dch8@st-andrews.ac.uk}
\thanks{DCH was financially supported by the EPSRC Doctoral Training Grant (EP/N509759/1).}

\author{Han Yu}
\address{Han Yu\\
	School of Mathematics \& Statistics\\University of St Andrews\\ St Andrews\\ KY16 9SS\\ UK  }
\curraddr{}
\email{hy25@st-andrews.ac.uk}
\thanks{HY was financially supported by the University of St Andrews.  }

\subjclass[2010]{Primary: 28A80, secondary: 11B13.}

\keywords{sumset, Assouad dimension, box dimension, Hausdorff dimension, distance set.}

\date{}

\dedicatory{}

\begin{abstract}
We study dimensions of sumsets and iterated sumsets and provide natural conditions which guarantee that a set $F \subseteq \mathbb{R}$ satisfies $\uBoxd F+F > \uBoxd F$ or even $\Haus n F  \to 1$. Our results apply to, for example, all uniformly perfect sets, which include Ahlfors-David regular sets.  Our proofs rely on Hochman's  inverse theorem for entropy and the Assouad and lower dimensions play a critical role.  We give several applications of our results including an Erd\H{o}s-Volkmann  type theorem for semigroups and new lower bounds for the box dimensions of distance sets for sets with small dimension.
\end{abstract}

\maketitle
\section{Introduction}\label{intro}

Studying the behaviour of sets under addition and multiplication with themselves has been of interest for many years, providing a multitude of fascinating results. Given $F \subseteq \mathbb{R}^d$, we are interested in relating the `size' of the \emph{sumset} $F+F = \{ a_1+a_2 \, \vert \, a_1,a_2 \in F\}$ and \emph{iterated sumsets}
\[
nF= F+F+\cdots+F= \left\{ a_1+a_2+\cdots+a_n \, \vert \,  a_i\in F, \forall \, i \in \left\{ 1,2,\ldots, n \right\}\right\} \qquad (n \geq 2)
\]
with the `size' of $F$.  When $F$ is finite, one interprets `size' as cardinality and the question falls under additive combinatorics, see \cite{TV} for an extensive introduction. We will also be interested in inhomogeneous sumsets $F+G = \{ a_1+a_2 \, \vert \, a_1 \in F, a_2 \in G\}$ and inhomogeneous iterated sumsets $F_1+F_2+\cdots+F_n= \left\{ a_1+a_2+\cdots+a_n \, \vert \,  a_i\in F_i, \forall \, i \in \left\{ 1,2,\ldots, n \right\}\right\}$.

If $F$ is infinite, then `size' can be interpreted as `dimension', and many natural questions arise.  For $F \subset \mathbb{R}$ one might na\"ively expect that `generically' $\dim nF = \min\{1, n \dim F\}$ or that at least $\dim nF \to 1$ as $n \to \infty$, provided $\dim F >0$, or that $\dim F+F > \dim F$, provided $\dim F \in (0,1)$.  However, these na\"ive expectations  certainly do not hold in general.    K\H{o}rner \cite{Kor} and Schmeling-Shmerkin \cite{SS} proved that for any increasing sequence $\{\alpha_n\}_{n=1}^\infty$ with $0\le \alpha_n \le 1$ for all $n$, there is a set $E\subset \mathbb{R}$ such that $\Haus nE = \alpha_n$ for all $n\ge 1$. This set can also be made to have specific upper and lower box dimensions $\left\{\beta_n\right\}$ and $\left\{\gamma_n \right\}$ given certain technical restrictions on these sequences. Schmeling and Shmerkin construct explicit sets with these properties. The main purpose of this paper is to identify natural conditions on $F$ which guarantee that the sumsets behave according to the  na\"ive expectations described above.

A related problem is the Erd\H{o}s-Volkmann ring conjecture which states that any Borel subring of $\mathbb{R}$ must have Hausdorff dimension either 0 or 1. This was solved by Edgar and Miller \cite{edgmill} where they not only showed that   a Borel subring   $F$ of $\mathbb{R}$ must have Hausdorff dimension either 0 or 1, but also if $\Haus F = 1$ then $F=\mathbb{R}$. Edgar and Miller also showed that any Borel subring   $F\subseteq \mathbb{C}$ has Hausdorff dimension 0, 1 or 2. On a related note, Erd\H{o}s and Volkmann \cite{erdvolk} proved that for every $0\le s \le 1$, there is an additive Borel subgroup   $G(s) \le \mathbb{R}$ such that $\Haus G(s) = s$. Therefore the fact that rings have both an additive and multiplicative structure is essential in obtaining the dimension dichotomy.

One can also consider specific classes of sets and hope to get stronger results concerning their sumsets.   Indeed, one of the main inspirations for this work was a result of Lindenstrauss, Meiri and Peres \cite{LMP}, which implies that for compact $\times p $ invariant subsets $F$ of the circle with $\Haus F > 0$, one has $\Haus n F \to 1$. This follows from a stronger result which states that if $\{E_i\}$ is a sequence of compact $\times p $ invariant  sets which satisfy
\[
\sum_i \frac{\Haus E_i}{|\log \Haus E_i |} = \infty,
\]
then $\Haus (E_1 + \cdots + E_n) \to 1$.  See our Corollary \ref{increasingcor} for a result related to this.

Recent work by Hochman \cite{Hoch, Hoch2, Hoch3} has used  ideas from additive combinatorics and entropy to make important contributions to the dimension theory of self-similar sets, in particular the overlaps conjecture, see \cite{PSol}. The techniques in our proofs will use some of the ideas developed by Hochman which will be summarised in Section \ref{hochman}.

In this paper we will consider several different dimensions, namely  the Hausdorff,  box,  Assouad  and  lower dimensions. We define all of these dimensions here except for the Hausdorff dimension, since we will not use this definition directly. For a definition of Hausdorff dimension and more information on the box dimension one can check \cite{Fa}. For any bounded set $E\subset \mathbb{R}^d$, we define $N(E,r)$ to be the smallest number of dyadic cubes of side lengths $r > 0$ needed to cover $E$. The \textit{upper box dimension} of a set $F \subset \mathbb{R}^d$ is defined to be 
\[
\uBoxd F = \limsup_{r \rightarrow 0} \frac{\log N(F,r)}{-\log r }
\]
and the lower box dimension $\lBoxd F$ is found by taking the liminf. When these limits coincide we simply talk about the box dimension $\Boxd F$.  The definitions of the box dimensions described above only apply for bounded sets, since for unbounded sets the covering number is always infinite. However, we modify the definition for convenience as follows. The upper and lower box dimensions of an unbounded set $F\subset\mathbb{R}$ are defined to be
	\[
	\uBoxd F=\sup_{K\subset F : K \text{ bounded }} \uBoxd K
	\]
and
	\[
	\lBoxd F=\sup_{K\subset F :  K \text{ bounded }} \lBoxd K.
	\]
	This definition also applies to bounded sets as well and in this case it clearly coincides with the usual definition.

For any set $F\subseteq \mathbb{R}^d$, the \textit{Assouad dimension} of $F$ is 
\begin{align*}
\Assouad F = \inf \Bigg\{ s \ge 0 \, \, \colon \, (\exists \, C >0)\, (\forall & R>0)\,  (\forall r \in (0,R))\, (\forall x \in F) \\ 
&N(B(x,R) \cap F,r) \le C \left( \frac{R}{r}\right)^s \Bigg\}
\end{align*}
where $B(x,R)$ denotes the closed ball of centre $x$ and radius $R$. Similarly the \textit{lower dimension} is 
\begin{align*}
\Lower F = \sup \Bigg\{ s \ge 0 \, \, \colon \, (\exists \, C >0)\, (\forall & R \in (0, \text{diam}(F))\,  (\forall r \in (0,R))\, (\forall x \in F) \\ 
&N(B(x,R) \cap F,r) \ge C \left( \frac{R}{r}\right)^s \Bigg\}
\end{align*}
where $\text{diam}( \cdot)$ denotes the diameter of a set. In order to force the lower dimension to be monotone, one often considers the \emph{modified lower dimension} $\dim_{\textrm{ML}} F = \sup \{ \Lower E  :  E \subseteq F\}$.  We omit further discussion of this but point out that throughout this paper one may replace lower dimension by modified lower dimension simply by working with subsets.  For further details concerning the Assouad and lower dimensions, we suggest \cite{F, Lu, R} for a general introduction. Roughly speaking Assouad dimension  provides information on how `locally dense' the set can be whilst the lower dimension tells us how `locally sparse' it can be. One of the main themes of this paper is that these notions turn out to be critical in the study of sumsets. 
 It is useful to keep in mind that for any set $F$
\[
\Haus F \leq \lBoxd F \leq \uBoxd F  \leq  \Assouad F \qquad \text{and} \qquad \Lower F \leq \lBoxd F
\]
and if $F$ is closed, then one also has $ \Lower F \leq \Haus F$.

\section{Results}\label{results}

\subsection{Dimension growth for sumsets and  iterated sumsets}

We first derive general conditions which force the dimensions of the sumset to strictly exceed the dimensions of the original set.    It follows from recent work of Dyatlov and Zahl (private communication, see also \cite{Dy}) that if $F\subset \mathbb{R}$ is Ahlfors-David regular with dimension strictly between 0 and 1, then $\uBoxd F < \uBoxd 2F$ (this is even true for lower box dimension).   This result can be interpreted as `regularity implies dimension growth'. If a set is Ahlfors-David regular, then the lower, Hausdorff, box and Assouad dimensions all coincide and, as such, our results below apply to a much larger class of sets where Ahlfors-David regularity is weakened to only requiring that either the lower dimension is strictly  positive or the Assouad dimension is strictly less than 1.  This is natural since, for example,  sets with Assouad dimension strictly less than 1 are precisely the sets which  uniformly avoid arithmetic progressions \cite{FY, kota}, and arithmetic progressions tend to cause the sumset to be small.

\begin{thm}\label{main1asym}
Let $F_1,F_2 \subset \mathbb{R}$ with $\uBoxd F_1, \lBoxd F_2  \in (0,1)$.  If either $\Assouad F_1 < 1$ or $ \Lower F_2 >0$, then
\[
  \uBoxd F_1 < \uBoxd (F_1+F_2).
\]
\end{thm}

This theorem will be proved in Section \ref{main1proof} and the proof will rely on the inverse theorem of Hochman as described in Section \ref{hochman}. We learned after writing this paper that the Assouad dimension part of this result can be derived from \cite[Theorem  5]{Hoch4}, which is stated in terms of measures.   We obtain the following corollary in the symmetric case.

\begin{cor}\label{main1sym}
Let $F \subset \mathbb{R}$ with $0< \uBoxd F  < 1$.  If either $\Assouad F < 1$ or $ \Lower F >0$, then
\[
\uBoxd F < \uBoxd 2F.
\]
\end{cor}

 Notice that we only need the upper box dimension condition here  and so the result is not a direct corollary of the statement above.  However, a careful check of the proof shows that if the two sets are the same, then only information about the upper box dimension is required.  This will be commented on during the proof of Theorem \ref{main1asym}.

We also obtain a corollary about sumsets of sequences of sets which should be compared to the result of Lindenstrauss, Meiri and Peres concerning $\times p$ invariant sets mentioned in the introduction.

\begin{cor}\label{increasingcor}
Let $\{E_i\}$ be a sequence of subsets of $\mathbb{R}$ which satisfy $\Lower E_i > 0$ for all $i$.  Then  $\uBoxd (E_1 + \cdots + E_n)$ forms a strictly increasing sequence in $n$ unless it reaches 1, in which case it becomes constantly equal to 1 from then on.
\end{cor}

\begin{proof}
This follows immediately from Theorem \ref{main1asym} where for each $n$ we take $F_1 = E_1 + \cdots + E_n$ and $F_2 = E_{n+1}$.
\end{proof}

Corollary \ref{increasingcor} is stronger than the result of Lindenstrauss, Meiri and Peres in that the sets $E_i$ need not be dynamically invariant, and the assumption $\Lower E_i > 0$ for all $i$ allows  $\Haus E_i $ to converge to 0 at  any rate.  However, it is also weaker since we obtain a much weaker form of dimension growth: strict increase rather than convergence to 1.

Following  \cite{Fr}, we obtain an Assouad dimension version of Corollary \ref{main1sym} by passing the problem to the level of tangents.
\begin{cor}\label{main1cor}
Let $F \subseteq \mathbb{R}$.  If $0<\Assouad F < 1$, then
\[
\Assouad F < \Assouad 2F.
\] 
\end{cor}
This corollary will be proved in Section \ref{main1corproof}. Corollary \ref{main1cor} is particularly interesting because it is a statement only about the Assouad dimension and is false if Assouad dimension is replaced by Hausdorff, or upper or lower box dimension, due to the examples in \cite{SS}.
\begin{rem}    
Similar results actually hold for $F-F$ instead of $2F$. To see this, it is sufficient to observe that $F$ and $-F$ have the same associated tree, $T$, up to reflection, where associated trees will be defined in Section \ref{hochman}.
\end{rem}

Next we  derive general conditions which force the dimensions of the iterated sumset to approach 1 in the limit. 

\begin{thm}\label{main2}
Let $F\subseteq \mathbb{R}$.  If $\Lower F > 0$, then
\[
\lim_{n\rightarrow\infty}\Lower nF = 1.
\]
In particular, if $F$ is closed, or even if $F$ has a closed subset with positive lower dimension, then
\[
\lim_{n\rightarrow\infty}\Haus nF = 1.
\]
\end{thm}

This theorem will be proved in Section \ref{main2proof}, again relying on Hochman's inverse theorem.  Note that since the lower dimension is a lower bound for lower and upper box dimension and Assouad dimension we see that $\dim nF \to 1$ for these dimensions also. Theorem \ref{main2} applies to Ahlfors-David regular sets with dimension strictly between 0 and 1 and therefore answers a question posed to us by Josh Zahl by showing that the Hausdorff dimension of iterated sumsets of Ahlfors regular sets approaches 1.    Corollary \ref{increasingcor}  shows that,  in the setting of Theorem \ref{main2},  $\uBoxd nF$ is strictly increasing in $n$ while it is less than 1.    Theorem \ref{main2}  should also be compared with the results in \cite{LMP}, in particular the corollary discussed in our introduction concerning homogeneous iterated sumsets. 

There exist sets of zero lower dimension and positive Hausdorff dimension for which the box dimension of the  iterated sumsets does not approach 1, see \cite{SS}. Thus Theorem \ref{main2} is sharp in the sense that lower dimension cannot be replaced by one of the other dimensions discussed in this paper. We note that the Assouad dimension of the set does not influence Theorem \ref{main2}.  The work of Astels \cite{astels} is related to Theorem \ref{main2}. In particular, \cite[Theorem 2.4]{astels} proves that if a Cantor set $C$ satisfies a  certain `thickness condition', then $nC$ contains an interval for some $n$.

If a set has positive \emph{Fourier} dimension then the Hausdorff dimension of the iterated sumset will approach 1 (in fact it will contain an interval after finitely many steps,  see \cite[Proposition 3.14]{Ma}). However, lower dimension and Fourier dimension are incomparable and  deterministic examples of sets with positive Fourier dimension are somewhat rare.  For example,  being Ahlfors-David regular does not imply positive Fourier dimension but does imply positive lower dimension.  Indeed, the middle third Cantor set is well-known to have Fourier dimension 0. However, sets with positive lower dimension (or at least a subset with positive lower dimension) are more prevalent.  For example,  uniformly perfects has positive lower dimension.  Such sets include self-similar sets, self-conformal sets, self-affine sets, and limit sets of geometrically finite Kleinian groups.  

Hochman \cite{Hoch3} has also extended the inverse theorems to higher dimensions. This provides a platform for us to generalise our results on sumsets to higher dimensions, but we do not pursue the details. The same approach and arguments apply, but the results are slightly different to accommodate the higher dimensional phenomenon that dimension can get `trapped' in a subspace.

\subsection{An Erd\H{o}s-Volkmann type theorem for semigroups}

In Section \ref{intro} we briefly mentioned a  dichotomy for the Hausdorff dimension of Borel subrings of $\mathbb{R}$ (it can only be 0 or 1). This dichotomy fails for subgroups, but if we consider the box dimension instead, a similar dichotomy holds. In fact, if $F\subset\mathbb{R}$ is an additive group then $F$ is dense in $\mathbb{R}$ or $F$ is uniformly discrete.  We say a set is \emph{uniformly discrete} if $\inf|x-y|>0$ where the infimum is taken over all pairs of distinct elements $x,y$ in the set.   We recall that a dense set has full box dimension whilst a uniformly discrete set has 0 box dimension, even when unbounded. 

A natural extension of this kind of problem is to remove even more structure and so we consider additive \emph{semi}groups. (Nonempty) semigroups can of course be uniformly discrete, e.g. $\mathbb{Z}$ or $\mathbb{N}$,  or dense, e.g. $\mathbb{Q}$, but there are three further possibilities:
\begin{itemize}
		\item[(1)]  the semigroup is somewhere dense, but not dense, e.g. $[1,\infty) \cap \mathbb{Q}$ or $(-\infty, -2) \cup \{-1\}$,
\item[(2)]  the semigroup is discrete, but not uniformly discrete, e.g. the semigroup generated by $\{1, \alpha\}$ where $\alpha>0$ is irrational,
\item[(3)] the semigroup is nowhere dense, but not discrete, e.g. the semigroup generated by the set $\{2-1/n : n \in \mathbb{N}\}$.
        \end{itemize}
Note that in the three `new' cases, the semigroup is necessarily contained in either $[0, \infty)$ or $(-\infty, 0]$.  The only interesting case from a dimension point of view is (3), noting that in case (1) the box dimensions are trivially 1 and in case (2) they are trivially 0. In case (3) we obtain the following result as a consequence of our main results.

\begin{cor}\label{cor1}
If  $F\subset \mathbb{R}$ is an additive semigroup with  $\uBoxd F \in (0,1)$, then $\Lower F  = 0$ and at least one of the following holds:
        \begin{itemize}
        \item[(i)] $\Assouad F \cap I= 1$ for some bounded interval $I\subset \mathbb{R}$
        \item[(ii)] $\uBoxd F \cap [-2^{n},2^{n}] < \uBoxd F \cap [-2^{n+1},2^{n+1}]$ for all sufficiently large integers $n$.
        \end{itemize}       
\end{cor}

Note that every additive subsemigroup of $\mathbb{R}$ (apart from $\{0\}$ and $\emptyset$)  contains an infinite arithmetic progression and therefore has full Assouad dimension, see \cite{FY}, and so the interest of the conclusion  (ii) is that this dimension is obtained in a bounded component.  Also note that additive semigroups with  $\uBoxd F \in (0,1)$ exist and can be constructed by generating a semigroup by a suitable translate of one of the sets $E$ constructed by Schmeling-Shmerkin \cite{SS} for which $\uBoxd nE$ does not approach 1, but $\uBoxd E >0$.

\begin{proof}
The fact that $\Lower F  = 0$   follows immediately from Theorem \ref{main2} since $nF \subset F$ for all $n$ and so $\Lower F  >0$ would guarantee that $\uBoxd F = 1$.  Assume without loss of generality that   $F \subseteq [0,\infty)$   and decompose $F$ as follows
	\[
	F\cap [0,\infty)=(F\cap [0,1]) \cup \left(\bigcup_{i\geq 0} F\cap [2^i,2^{i+1} ]\right)=\bigcup_{i\geq -1} F_i,
	\]
where $F_{-1}=F\cap [0,1]$ and $F_{i}=F\cap [2^i,2^{i+1}]$ for $i\geq 0$. We denote the partial union by $G_k=\bigcup_{i = -1}^{k} F_i$. Since $F$ is a semigroup, $ 2G_k \subset F_{k+1}\cup G_k $ and therefore
	\[
	\uBoxd 2G_k\leq \max\left\{\uBoxd G_k,\uBoxd F_{k+1}\right\}. 
	\]
	By Corollary \ref{main1sym} we see that either $\uBoxd G_k=0$ or $\Assouad G_k=1$ or $\uBoxd F_{k+1}>\uBoxd G_{k}$. Since we assume $\uBoxd F \in (0,1)$, there exists an integer $k_0$ such that $\uBoxd G_k \in (0,1)$ for all $k \geq k_0$.  Therefore, either $\Assouad G_{m}=1$ for some $m$ in which case we are in (i) and can choose $I = [0,2^{m}]$ or $\Assouad G_m< 1$ for all $m$, in which case $\uBoxd F \cap [0,2^{n}] < \uBoxd F \cap [0,2^{n+1}]$ for all $n \geq k_0$ and we are in case (ii).
\end{proof}

\subsection{Dimension estimates for distance sets}

Sumsets $F+F$ are related to difference sets $F-F$ and distance sets $|F-F|$ and so we can use our techniques to get results for these sets too. The distance set of $F \subset \mathbb{R}^d$ is
\[
D(F)=\{|x-y|\colon x,y\in F\}.
\]
For example, it follows immediately from Corollary \ref{main1cor} that for a set $F\subset \mathbb{R}$ of Assouad dimension strictly between 0 and 1 we have
\[
\Assouad D(F) > \Assouad F.
\]
Geometric properties of distance sets have been studied extensively with much effort focusing on Falconer's distance set conjecture, which stemmed from  \cite{falconerdistance}. One version of this asserts that if a Borel set $F\subset\mathbb{R}^d$ has Hausdorff dimension strictly larger than $d/2$, then the distance set should have positive Lebesgue measure. A related problem, concerning dimension only is as follows.

\begin{conj}[Falconer's conjecture]
Let $\dim$ denote one of the Hausdorff, packing, box or Assouad dimensions.	If $F\subset\mathbb{R}^d$ satisfies $\dim F>d/2$, then $\dim D(F) = 1$.
\end{conj}

The above conjecture has been proved for Ahlfors-David regular sets in $\mathbb{R}^2$ for packing dimension \cite{Or} and, more recently, for Hausdorff dimension for Borel sets in $\mathbb{R}^2$ with equal Hausdorff and packing dimension \cite{Sch}.  It has also been resolved in $\mathbb{R}^2$ for the Assouad dimension \cite{Fr}.

Instead of looking for a condition ensuring the distance set has full dimension, we  obtain a lower estimate for the dimension of the distance set as a function of the dimension of the original set. We also restrict ourselves to the Assouad and upper box dimension of sets for this section. A recent result by Fraser \cite{Fr} provides lower bounds for the Assouad dimension of the distance set for sets of large Assouad dimension. The following result complements these bounds by providing lower bounds for sets with small dimension.

\begin{thm}\label{distancethmA}
	Let $F\subset\mathbb{R}^d$ be such that  $0<\Assouad F<d$.  Then
	\[
	\Assouad D(F)>\frac{\Assouad F}{d}.
	\]
\end{thm}

This bound is new for sets with small Assouad dimension and for sets with large Assouad dimension the bound
\[
	\Assouad D(F)\geq\max\left\{\frac{6\Assouad F+2-3d}{4},\Assouad F-\frac{d-1}{2}\right\}
	\]
from \cite{Fr} is better, see Figure \ref{lowerplot} below for a depiction of the case when $d=3$.

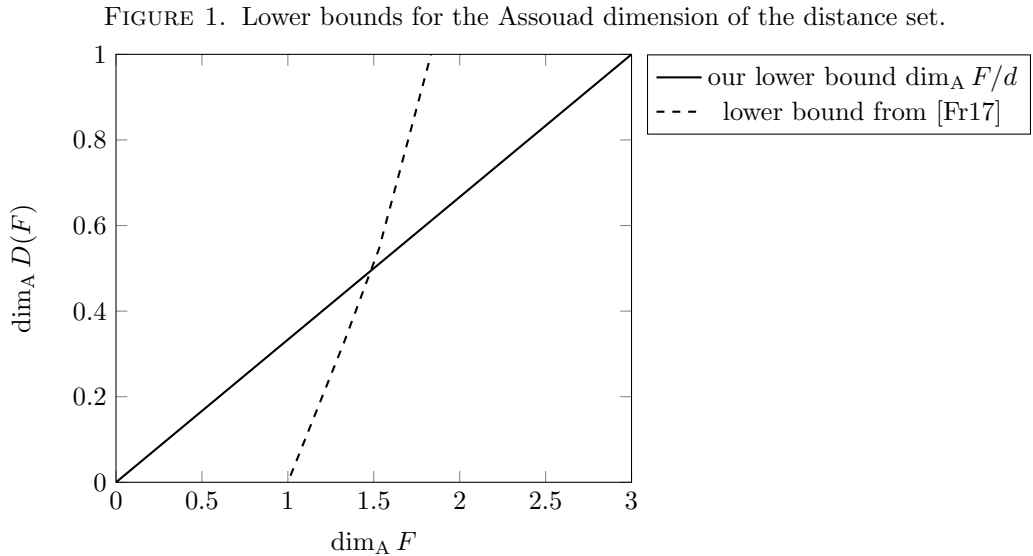
\begin{figure}[h] 
	\caption{Lower bounds for the Assouad dimension of the distance set.}
\begin{centering}
\begin{tikzpicture}
[
  declare function={
    func(\x)= (\x < 1.5) * (\x-1)   +
              (\x >= 1.5) * (6/4*\x-7/4)
   ;
  }
]
\begin{axis}[xlabel = $\Assouad F$,ylabel = $\Assouad D(F)$,xmax=3,xmin=0,ymin=0,ymax=1, samples=50,legend pos=outer north east]
  \addplot[black, thick] (x,x/3);
  \addplot[dashed][black,   thick] {func(x)};
  \addlegendentry{our lower bound $\Assouad F /d$}
  \addlegendentry{lower bound from \cite{Fr}}
\end{axis}
\end{tikzpicture}
\end{centering}
\label{lowerplot}
\end{figure}

We also obtain a similar result for the upper box dimension.
 
\begin{thm}\label{distancethmB}
For $F\subset \mathbb{R}^d$, we have
\[
\uBoxd D(F)  \geq \frac{\uBoxd F}{d}.
\]
Moreover, if $\uBoxd F > 0$ and $\Assouad D(F) < 1$, then
\[
\uBoxd F < ( d - 1)  \uBoxd D(F) +  \Assouad D(F).
\]
\end{thm}
Depending on the Assouad dimension of the distance set, the second inequality can be better or worse than the first one. Curiously, if one considers the distance set with respect to the supremum norm, where the unit ball is a square, our methods show that if $\uBoxd F > 0$  and $\Assouad D(F)<1$, then
\[
\uBoxd D(F)  > \frac{\uBoxd F}{d}.
\]

\section{Hochman's inverse theorem and entropy}\label{hochman}
To properly state Hochman's inverse theorem some definitions are needed, notably entropy and the uniformity and atomicity of measures. Thereafter several technical lemmas relating entropy and covering numbers will be discussed. 

\begin{defn}[Dyadic intervals and restrictions of measures]\label{DY}
	For any integer $n\geq 0$, the set of level $n$ dyadic intervals is
	\[
	\mathcal{D}_n=\left\{D_n(k)=[k2^{-n},(k+1)2^{-n}): k\in\mathbb{N}, 0\leq k\leq 2^n-1\right\}.
	\]
	
	For any measure $\mu$ on the real line, $x\in [0,1)$ and $n\in\mathbb{N}$, define $D(x,n)$ to be the unique dyadic interval of level $n$ which contains $x$ and $T^{D(x,n)}$ to be the unique orientation preserving affine map taking $D(x,n)$ to $[0,1]$. For $x,n$ such that $\mu(D(x,n))>0$, we write
	\[
	\mu_{D(x,n)}=\mu_{x,n}=\frac{1}{\mu(D(x,n))}\mu \vert_{D(x,n)}
    \]
and
    \[
    \mu^{D(x,n)}=\mu^{x,n}=T^{D(x,n)}\mu_{D(x,n)}.
	\]
\end{defn}

We will use both $\mu^{D(x,n)}$ and $\mu^{x,n}$ interchangeably, often the choice of notation will be picked to emphasise the object studied, be it a point or an interval.

\begin{defn}[Entropy]
	Given a probability measure $\mu$ on $[0,1]$, we define the $n$-level entropy to be
	\[
	H(\mu,\mathcal{D}_n)=-\sum_{E\in \mathcal{D}_n} \mu(E)\log \mu(E),
	\] 
	where we assume $0\log 0$ to be $0$. The averaged $n$-level entropy is then defined to be
	\[
	H_n(\mu)=\frac{1}{n\log 2}H(\mu,\mathcal{D}_n).
	\]
\end{defn}
 
 \begin{defn}\label{uniform-atom}
 	For a probability measure $\mu$ on $[0,1]$ and two numbers, $\varepsilon\in [0,1]$ and $m \in \mathbb{N}$, we say that $\mu$ is $(\varepsilon,m)$-uniform if 
 	\[
 	H_m(\mu)\geq 1-\varepsilon.
 	\]
 	We say that $\mu$ is $(\varepsilon,m)$-atomic if
 	\[
 	H_m(\mu)\leq \epsilon.
 	\]
 \end{defn}

Hochman's inverse theorem can now be stated as introduced in \cite{Hoch}. This result and its proof are discussed in further detail in the survey \cite{Hoch2} and the lecture notes \cite{Sa}.

\begin{thm}[Theorem 4.11 \cite{Hoch}]\label{HochmanEntropy}
	For any $\varepsilon>0$ and integer $m$, there exists  $\delta=\delta(\varepsilon,m)$ and $n_0=n_0(\varepsilon,m,\delta)$ such that for any $n>n_0$ and any probability measures $\mu, \nu$ on $[0,1]$, either $H_n(\mu*\nu)\geq H_n(\mu)+\delta$ or there exist disjoint subsets $I,J\subset\{0,\dots ,n\}$ with $\#(I\cup J)\geq (1-\varepsilon)n$ and
	\[
	\mu (\{ x \in [0,1] :  \mu^{x,k} \text{ is } (\varepsilon,m)\textup{-uniform}\})>1-\varepsilon, \text{ if }k\in I
	\]
	\[
	\nu(\{ x \in [0,1] : \nu^{x,k} \text{ is } (\varepsilon, m)\textup{-atomic}\})>1-\varepsilon, \text{ if }k\in J.
	\]
\end{thm}

We wish to study sets, not measures. To do this we need to link the entropy of a measure to the covering number of the support of the measure. We will do this in two ways. The first idea is to find an analogous definition for $(\varepsilon,m)$-uniformity of a set. This is possible since compact subsets of $\mathbb{R}$ are in 1-1 correspondence with subsets of the full binary tree in a canonical way which we   describe below. The second will be to consider the covering number of a set supposing that the uniform measure is sufficiently full branching or atomic.

We identify $D_n(i)$ with the $i$th vertex at the $n$th level of the standard infinite binary tree (where we count vertices in a given level from left to right). Observe that if $D_n(i)\cap F\neq\emptyset$ then at least one of the dyadic intervals $D_{n+1}(2i)$ or $D_{n+1}(2i+1)$ intersects $F$ and all of the dyadic intervals containing $D_n(i)$ also intersect $F$. Therefore the vertices of the infinite binary tree for which $D_n(i) \cap F \neq \emptyset$ give rise to a subtree $T$ which describes the distribution of $F$. We say a dyadic interval $D_n(i)$ is a \emph{descendant} of another dyadic interval $D_m(j)$ if $D_n(i) \subset D_m(j)$ and pass this terminology to the vertices of $T$ by the above association. Similarly a vertex is a level $n$ vertex if it is associated with a dyadic interval $D_n(i)$ for some $i$.  We shall call $T$ the \emph{tree associated} with $F$ and denote it by $T_F$. Understanding properties of this tree will give us direct information about coverings of $F$ by dyadic intervals.

The analogue of $(\varepsilon,m)$-uniform in terms of our tree is the following. We say $T$ is \emph{$(\varepsilon,m)$-full branching at vertex $D_k(i)$} if $D_k(i)$ has at least $2^{(1-\varepsilon) m}$ descendants $m$ levels below, that is, $F$ intersects at least $2^{(1-\varepsilon) m}$ many level $k+m$ dyadic intervals contained in $D_k(i)$. An analogue of $(\varepsilon,m)$-atomic does exist, however it is not needed in this paper since we consider more regular sets when looking at $(\varepsilon,m)$-atomic measures and the measure will thus provide direct information about the covering number, see Lemma \ref{MoranLower}. 

We will now show how full branching for measures implies full branching for sets. To do so we need the following result, which will be used extensively in this article and can be found in \cite{CT}.

\begin{lma}\label{EntropyHigh}
	Let $A$ be a finite set then for any probability measure $\mu$ on $A$ we have the following inequality
	\[
	0\leq -\sum_{a\in A} \mu(\{a\})\log \mu(\{a\})\leq \log\#A.
	\]
	The maximal value is attained when $\mu$ is uniform on its support, that is
	\[
    \mu(\{a\})=
    \begin{cases}
    \frac{1}{\#A}  \quad &a\in A \\
    0 \quad &\text{otherwise}
    \end{cases}
	\]
	The minimal value $0$ is attained when $\mu$ is supported on a single point.
\end{lma}

Let $\varepsilon, m$ be as in Definition \ref{uniform-atom}, $x \in \mathbb{R}$ and $k\in \mathbb{N}$ and $i$ be such that $x\in D_k(i)$. If a measure $\mu$ is such that $\mu^{x,k}$ is $(\varepsilon,m)$-uniform then by definition $H_m(\mu^{x,k})\ge 1- \varepsilon$. So
\begin{align*}
-\sum_{D\in \mathcal{D}_m} \mu^{x,k}(D) \log \mu^{x,k}(D) &\ge m (1-\varepsilon)\log 2
\end{align*}
and by Lemma \ref{EntropyHigh}
\[
m(1-\varepsilon)\log 2 \le - \sum_{D\in \mathcal{D}_m} \mu^{x,k}(D) \log \mu^{x,k}(D) \le \log N(\text{supp}( \mu^{x,k} ), 2^{-m}).
\]
Thus $N(\text{supp}(\mu^{x,k}), 2^{-m}) \ge 2^{(1-\varepsilon)m}$ and the tree associated with the support of $\mu$ is $(\varepsilon,m)$-full branching at $D_k(i)$.

Thus high entropy implies high covering number. The other direction is in general not true. When $\mu$ is $(\varepsilon,m)$-atomic, $N(\text{supp}( \mu), 2^{-m})$ can be large. However, the measure at scale $2^{-m}$ must be very non-uniform. 

The following lemma is the key to our second idea, heuristically saying that if entropy is low (or large) on a sufficient portion of scales then the covering number of the whole set at one specific scale will be low (or large).

\begin{lma}\label{EntropyToCovering}[Entropy and covering number]
	Let $F$ be a $2^{-n}$-separated finite subset of $[0,1]$, $\varepsilon \in [0,1]$ and $m\in \mathbb{N}$. Let $\mu$ be the uniform probability measure on $F$ and suppose that  
\[
\mu (\{ x \in [0,1] :  \mu^{x,i} \text{ is } (\varepsilon,m)\textup{-atomic}\})>1-\varepsilon
\]
for all $i\in I\subset \{0,\dots,n\}$ with $\#I\geq (1-\varepsilon)n$. Then
\[
N(F,2^{-n})\leq 2^{5\varepsilon n}.
\]
Similarly, suppose that
\[
\mu (\{ x \in [0,1] :  \mu^{x,i} \text{ is } (\varepsilon,m)\textup{-uniform}\})>1-\varepsilon
\]
for $i\in J\subset \{0,\dots,n\}$ with $\#J\geq (1-\varepsilon)n$. Then we have
\[
N(F,2^{-n})\geq 2^{(1-\varepsilon)^3 n}.
\]
\end{lma}
\begin{proof}
	We assume $\mu$ satisfies the first condition and shall compute $H(\mu,\mathcal{D}_n)$. First notice that
	\[
	H(\mu,\mathcal{D}_n)=H(\mu,\mathcal{D}_0)+\sum_{i=0}^{n-1}(H(\mu,\mathcal{D}_{i+1})-H(\mu,\mathcal{D}_{i})).
	\]
	Write $H(\mu,\mathcal{D}_{i+1}|\mathcal{D}_{i})=H(\mu,\mathcal{D}_{i+1})-H(\mu,\mathcal{D}_{i})$ for the conditional entropy. When $i\in I$ we see that
	\begin{align*}
\sum_{j=i}^{i+m-1} H(\mu,\mathcal{D}_{i+1}|\mathcal{D}_{i})&=\int H(\mu^{x,i},\mathcal{D}_m)d\mu(x)\\
&=\int_{x : \, H(\mu^{x,i},\mathcal{D}_m)\leq \varepsilon}H(\mu^{x,i},\mathcal{D}_m)d\mu(x) +\int_{x: \, H(\mu^{x,i},\mathcal{D}_m)>\varepsilon} H(\mu^{x,i},\mathcal{D}_m)d\mu(x).	
	\end{align*}

	Notice that $H(\mu^{x,i},\mathcal{D}_m)\leq m\log 2$ for all $x$. Then we see that
	\[
	\sum_{j=i}^{i+m-1} H(\mu,\mathcal{D}_{i+1}|\mathcal{D}_{i})\leq \varepsilon\mu(\{x\colon H(\mu^{x,i},\mathcal{D}_m)\leq \varepsilon\})+m  \log 2 \left(\mu(\{x \colon H(\mu^{x,i},\mathcal{D}_m)> \varepsilon\}) \right)
	\]
	and by our assumption
    \[
    \varepsilon\mu(\{x\colon H(\mu^{x,i},\mathcal{D}_m)\leq \varepsilon \})+m\log 2  \left(\mu( \{ 
x \colon H(\mu^{x,i},\mathcal{D}_m )> \varepsilon \}) \right)
< \varepsilon+m\varepsilon\log 2.
    \]
    
	When $i\notin I$ we only have the following trivial bound
	\[
	H(\mu,\mathcal{D}_{i+1}|\mathcal{D}_{i})\leq \log 2.	
	\]
	Now we can cover $I$ with disjoint intervals of form $[i,i+m]$ for $i\in I$ by a greedy covering procedure. Let $i_1$ be the smallest number in $I$ and we pick the interval $[i_1,i_1+m].$ Then we choose the smallest number $i_2$ in $I$ which is larger than $i_1+m$ and we pick the interval $[i_2,i_2+m].$ We can iteratively apply the above argument until we have covered all elements in $I$. There are at most $n/m+1$ intervals needed in this cover. The cardinality of the uncovered subset of $[1,\dots,n]$ is bounded above by the cardinality of $[1,\ldots, n]\setminus I $, so is at most $\varepsilon n$.
	Therefore we see that
	\[
	H(\mu,\mathcal{D}_n)\leq \left(\frac{n}{m}+1\right) (\varepsilon+\varepsilon m\log 2)+\varepsilon n \log 2\leq 5\varepsilon n\log 2.
	\]
	As $\mu$ is uniform on $F$ and $F$ is $2^{-n}$ separated we see that
	\[
	\log N(F,2^{-n})=\log \#F=H(\mu,\mathcal{D}_n)\le 5\varepsilon n\log 2.
	\]
	Therefore
	\[
	N(F,2^{-n})\leq 2^{5\varepsilon n}.
	\]
	This proves the first part of the lemma.
    
    The second part can be proved in a similar manner by breaking the integral in the following slightly different way. For each $i\in J$ we have the following equality,
	\begin{align*}
	\sum_{j=i}^{i+m-1} H(\mu,\mathcal{D}_{i+1}|\mathcal{D}_{i})&=\int H(\mu^{x,i},\mathcal{D}_m)d\mu(x)\\
    &=\int_{x: H(\mu^{x,i},\mathcal{D}_m)\leq m(1-\varepsilon)}H(\mu^{x,i},\mathcal{D}_m)d\mu(x)+\int_{x: H(\mu^{x,i},\mathcal{D}_m)>m(1-\varepsilon)}H(\mu^{x,i},\mathcal{D}_m)d\mu(x).
	\end{align*}
	The first term on the right can be trivially bounded below by $0$ and the second term can be bounded from below by $(1-\varepsilon)^2m\log 2$. Then we can cover $J$ with disjoint intervals of the form $[i,i+m]$ with $i\in J$ as above, using at least $(1-\varepsilon) n /m$ intervals for this cover. From here the result follows since
	\[
	H(\mu,\mathcal{D}_n)\geq \frac{(1-\varepsilon) n}{m} (1-\varepsilon)^2m\log 2= (1-\varepsilon)^3n\log 2
	\]
and therefore $N(F,2^{-n}) \geq 2^{(1-\varepsilon)^3n}$, which completes the proof. 
\end{proof}
Finally note that we will often  consider finite approximations of sets. Given a set $F\subset [0,1]$ and an integer $n$, we define the $2^{-n}$ discretization of $F$ to be the following set
\[
F(n)=\left\{a :[a,b)\in \mathcal{D}_n \text{ and } [a,b) \cap F \neq\emptyset \right\}.
\]
Notice that $F(n)$ might not be a subset of $F$. However, their associated trees coincide up to level $n$ and $N(F(n),2^{-n})=N(F,2^{-n})$. Moreover, for two sets $F_1,F_2\subset [0,1]$, $F_1(n)+F_2(n)$ is $2^{-n}$ separated and 
\[
\frac{1}{2} N(F_1+F_2,2^{-n})\leq\#(F_1(n)+F_2(n))\leq 2 N(F_1+F+2,2^{-n}).
\]
Due to this, $\#(F_1(n)+F_2(n))$ is useful for estimating the box dimensions of $F_1+F_2$.

\section{Proofs}

We start by proving Theorem \ref{main1asym} in Section \ref{main1proof}, followed by  Corollary \ref{main1cor} in Section \ref{main1corproof}. In Section \ref{main2proof} we prove Theorem \ref{main2}.  In Sections \ref{distanceproofA} and \ref{distanceproofB} we will prove Theorems \ref{distancethmA} and \ref{distancethmB} respectively, which concern distance sets.  The final section of the paper discusses several examples, including Section \ref{selfsimsection} which handles various dynamically invariant sets.

\subsection{Proof of Theorem \ref{main1asym}: strict increase}\label{main1proof}

We break the proof down into a few lemmas, from which the conclusion of Theorem \ref{main1asym} immediately follows.

\begin{lma}\label{lma1}
	Let $F_1, F_2\subset [0,1]$ with $\uBoxd F_1+F_2 =\uBoxd F_1$, then either $\lBoxd F_2=0$ or $\Assouad F_1=1$.
\end{lma}

\begin{proof}
For the upper box dimension it is convenient to introduce \emph{observing scales} defined to be any sequence of real numbers $0<r_i<1$ such that
	\[
	\lim_{i\to\infty} \frac{\log N(F_1,r_i)}{-\log r_i}=\uBoxd F_1 \text{ and } \lim_{i\to\infty} r_i=0.
	\]
The existence of such sequences comes directly from the definition of upper box dimension. Moreover, we can assume the observing scales  are dyadic, that is, we can find a strictly increasing  integer sequence $n_i$ such that $2^{-n_i}$ are observing scales.  Fix a set of dyadic observing scales and let $\delta \in (0,1)$.  For all sufficiently large $i$, we have
\[
	\#(F_1+F_2)(n_i)\leq 2 N(F_1+F_2,2^{-n_i}) \leq N(F_1,2^{-n_i})^{1+\delta}=\#F_1(n_i)^{1+\delta}.
	\] 
If this were not true, then we would have $\uBoxd F_1+F_2\geq(1+\delta)\uBoxd F_1$ which contradicts our assumption that $\uBoxd F_1+F_2 = \uBoxd F_1$.  
	
Let $\varepsilon>0$ be arbitrary and choose $m=m(\varepsilon)=[\log 1/\varepsilon]$. (This choice of $m(\varepsilon)$ is not that important, in fact any function $f(\varepsilon)$ which monotonically goes to $\infty$ as $\varepsilon$ goes to $0$ will serve equally well.) Apply Theorem \ref{HochmanEntropy} to obtain a $\delta \in (0,1)$ and an $n_0\in \mathbb{N}$. Then for any $n_i \ge n_0$ we define the measures $\mu$ and $\nu$ to be the uniform counting measures on $F_1(n_i)$ and $F_2(n_i)$ respectively. Thus, if these measures satisfy the entropy condition in Theorem \ref{HochmanEntropy} then we can partition the levels $\{0,1,2, \dots ,n_i\}$ into sets $I,J$ and $K$ such that $\#(I\cup J)\geq (1-\varepsilon)n_i$ and the measures $\mu,\nu$ are as stated in Theorem \ref{HochmanEntropy}.

We wish to check the condition $H_{n_i}(\mu*\nu)\leq H_{n_i}(\mu)+\delta$ given $\#(F_1({n_i})+F_2(n_i))\leq \#F_1(n_i)^{1+\delta}$. As $\mu,\nu$ are uniform counting measures we see that
\[
H_{n_i}(\mu)=\frac{1}{n_i\log 2}\log \#F_1(n_i)
\]
\[
H_{n_i}(\nu)=\frac{1}{n_i\log 2}\log \#F_2(n_i).
\]
Then
\begin{align*}
H_{n_i}(\mu*\nu)&\leq \frac{1}{n_i\log 2} \log \#(F_1(n_i)+F_2(n_i)) \quad \text{ by Lemma \ref{EntropyHigh}} \\
&\leq \frac{1}{n_i\log 2} \log \#F_1(n_i)^{1+\delta} \\
&=H_{n_i}(\mu)+\delta H_{n_i}(\mu) \\
&\leq H_{n_i}(\mu)+\delta \qquad \text{since } H_{n_i}(\mu)\leq 1. 
\end{align*}    
If for all $n_i$ large enough, the set $I$ from the theorem is empty, then $\Boxd F_2$ will be very small because $\#J\geq (1-\varepsilon) n_i$ and we can apply Lemma \ref{EntropyToCovering} to $\nu$. This leads to
\begin{equation} \label{treecountingbd}
	N(F_2,2^{-n_i})\leq 2^{5\varepsilon n_i}.
\end{equation}
It follows that 
\[
	\lBoxd F_2\leq 5\varepsilon.
\]
Since $\varepsilon>0$ can be chosen arbitrarily small we conclude that $\lBoxd F_2 = 0$.  Note that we only get information about the lower box dimension here since the scales $2^{-n_i}$ were chosen to be observing scales for $F_1$, not $F_2$. If $F_1=F_2=F$ then  we can deduce $\uBoxd F=0$. This is needed to obtain Corollary \ref{main1sym}. 

Therefore, if  $\lBoxd F_2 > 0$, then for all $\varepsilon>0$ small enough and $m=[\log 1/\varepsilon]$ there is a $k\in \left\{0,\ldots, n\right\}$ (where $n$ is some large integer) and an $x\in [0,1]$ such that $\mu^{x,k}$ is $(\varepsilon,m)$-uniform. This then implies that there exists a $(\varepsilon,m)$-full branching subtree of length $m$ somewhere in $T_1$ by our discussion in Section \ref{hochman} and this clearly implies that $\Assouad F_1=1$. 
\end{proof}

We wish to show a dual result for the lower dimension. In the previous proof we relied on large entropy implying large covering number. As already mentioned, small entropy does not necessarily imply a small covering number. However if the set is sufficiently homogeneous then this is true. 

In order to tackle this problem we make the following observation: sets with positive lower dimension contain nearly homogeneous subsets. We start by introducing the following version of Moran constructions. Let $k$ be a positive integer. We first take the unit interval $[0,1]$ as our zeroth generation. Then for the first generation we take $k$ disjoint intervals $I_i$ all of length $l_1>0$ such that the distance between the intervals is at least $l_1$. For the second generation, we take each $I_i$ from the first generation and split it into $k$ disjoint intervals all of length $l_2$ with separation $l_2$ as well. We do this construction for a sequence of positive numbers $\{l_n\}_{n\in\mathbb{N}}$ and in the end we obtain a compact set $F\in [0,1]$ which is the intersection of all intervals from all generations. We call such $F$ Moran constructions with strong separation condition and uniform branching number $k$.

\begin{lma}\label{LowerDimensionSet}
	Let $F\subset [0,1]$ be compact with $\Lower F=s>0$. Then for any $\varepsilon>0$, we can find a subset $F'\subset F$ which is a Moran construction with strong separation condition and uniform branching number and $\Lower F'\geq s-\varepsilon$.
\end{lma}
\begin{proof}
	As $\Lower F = s$, we can find an integer $m$ such that for all $x\in F$ and all pairs of numbers $R,r$ with $0<r<2^m r\leq R<1$ we have the following inequality
	\[
	N(B(x,R),r)\geq \left(\frac{R}{r}\right)^{s-\varepsilon}.
	\]
	That is to say, the binary tree $T$ associated with $F$ has the property that any full subtree $T'$ of height $m$ contains at least $2^{(s-\varepsilon)m}$	many level $m$ vertices. A subtree $T'$ is full if it is maximal in the sense that we can not join any new vertex from $T$ to $T'$  without increasing the height of $T'$. 
	
We now construct a Moran construction inside $F$. For the first step we start at the root of $T$ and take the full subtree of length $m$ from that vertex. By dropping at most half of the vertices we can assume that the associated dyadic intervals are $2^{-m}$-separated. Then we can take any collection of $\lfloor2^{(s-\varepsilon)m-1}\rfloor$ level $m$ vertices and iterate this procedure on all chosen vertices. We can continue this process, and the resulting subtree $T'$ of $T$ is regular in the sense that any subtreee of $T'$ of height $m$ has roughly $2^{(s-\varepsilon)m-1}$ level $m$ vertices. The tree $T'$ is associated to a set $F'$ in the previously described way. $F$ is compact so closed and thus $F' \subset F$. Then it is easy to see that $F'$ has lower dimension at least $s-\varepsilon$ and it is a Moran construction with strong separation condition and uniform branching number.
\end{proof}

One can see that all the dimensions considered in this paper coincide for Moran sets but more information is needed. The following lemma will formalise the homogeneity of Moran constructions. 
\begin{lma}\label{MoranLower}
	Let $F\subset [0,1]$ be a Moran construction with strong separation condition and uniform branching number of positive lower dimension. Then there is a probability measure $\nu$ supported on $F$ and numbers $\varepsilon>0, m>0$ such that for all $x\in F, i\in\mathbb{N}$, $\nu^{x,i}$ is not $(\varepsilon,m)$-atomic.
\end{lma}
\begin{proof}
	Let $F$ be a Moran construction of dimension $s>0$ and assign mass one to $F\cap [0,1]$. We then split the measure equally between $[0,1/2]\cap F$ and $[1/2,1]\cap F$ so if $F$ intersects both halves then $F\cap [0,1/2]$ has measure $1/2$ but if $F\cap[0,1/2] = \emptyset$ then the whole measure is on $F\cap[1/2,1]$. This procedure is iterated over all dyadic intervals, equally splitting the mass of any dyadic interval between its descendants that intersect $F$. This procedure produces a measure $\nu$ on $F$. We shall now show that $\nu$ has the required property.
	
	Let $T$ be the tree associated with $F$. Let $\varepsilon > 0$ be small and $m$ be a large integer. We can find a constant $C>0$ such that for any vertex $a$ of $T$ and integer $n\geq m$, the number of descendants at level $n$ is bounded between $C^{-1} 2^{sn}$ and $C 2^{sn}$. This follows from the Moran construction. Also when $m$ is large, $C$ can be chosen close to $1$. Then due to the construction of $\nu$ we see that there exist $m,\varepsilon$ such that the level $m$ entropy of $\nu^{x,i}$ is $sm\log 2$. Thus $\nu^{x,i}$ is not $(\varepsilon,m)$-atomic for all $x\in F,i\in\mathbb{N}$.
\end{proof}

We are now able to prove the final lemma. The proof will follow the proof of Lemma \ref{lma1} with the added Moran construction needed for more control in the final step.

\begin{lma}\label{lma1lower}
	Let $F_1, F_2\subset [0,1]$ with $\uBoxd F_1+F_2 =\uBoxd F_1$, then either $\uBoxd F_1=1$ or $\Lower F_2=0$.
\end{lma}

\begin{proof}
  We can assume that $F_1$ and $F_2$ are compact. If not, we can take the closure and the Assouad, box and lower dimensions will not change. Also it is easy to see that the closure of $F_1+F_2$ is the same as the sumset of the closures of $F_1$ and $F_2$. Assume $\uBoxd F_1 + F_2 = \uBoxd F_1$ and $\Lower F_2 > 0$, then we want to show that $\uBoxd F_1 = 1$. Furthermore by Lemma \ref{LowerDimensionSet} we assume that $F_2$ is a Moran construction with strong separation condition and uniform branching number. Any Moran construction subset $F_2'$ of $F_2$ satisfies our assumptions:
  \[
  \uBoxd F_1 \le \uBoxd F_1 + F_2' \le \uBoxd F_1 + F_2 = \uBoxd F_1
  \]
  and
  \[
  \Lower F_2' > 0.
  \]
  Thus if we can show $\uBoxd F_1 = 1$ when $F_2$ is a Moran construction then the result will follow for any set $F_2$ of positive lower dimension.
  
  Fix a set of dyadic observing scales $2^{-n_i}$ for $F_1$ as before and let $\delta \in (0,1)$ which can be chosen arbitrarily. We can conclude that for all sufficiently large $i$, we have
\[
	\#(F_1+F_2)(n_i)\leq 2 N(F_1+F_2,2^{-n_i}) \leq N(F_1,2^{-n_i})^{1+\delta} = \#F_1(n_i)^{1+\delta}.
	\]	
Let $\varepsilon>0$ be arbitrary, $m=m(\varepsilon)=[\log 1/\varepsilon]$ and apply Theorem \ref{HochmanEntropy} to obtain constants $\delta=\delta(\varepsilon,m)$ and $n_0$. Using the same method as in Lemma \ref{lma1} we can show that the entropies of the uniform measure $\mu$ on $F_1(n)$ and the measure $\nu$, constructed in Lemma \ref{MoranLower}, on $F_2(n)$ satisfy the conditions for the inverse theorem. Thus, for $n_i$ large enough there is a partition of $\left\{0,\ldots, n_i \right\}$ into sets $I,J$ and $W$ with the properties stated in  Theorem \ref{HochmanEntropy}.
    
If for large enough $n_i \ge n_0$ the set $J$ from the theorem is empty, then $\uBoxd F_1$ should be very large because in this case $\#I\geq (1-\varepsilon) n_i$ and so `most' measures $\mu^{x,k}$, for $x\in [0,1]$ and $k\in I$, will be $(\varepsilon,m)$-uniform. Then by Lemma \ref{EntropyToCovering} we deduce that
\[
N(F_1,2^{-n_i})\geq 2^{(1-\varepsilon)^3n_i}.
\]
It follows that 
\[
	\uBoxd F_1\geq (1-\varepsilon)^3  \to 1
	\]
as $\varepsilon \to 0$ and hence $\uBoxd F_1 = 1$. 

Therefore, if  $\uBoxd F_1 <1$, then  for all $\varepsilon>0$ small enough and $m=[\log 1/\varepsilon]$ there exists  $x \in [0,1]$ and $k\in \left\{0,\ldots, n \right\}$ (for some large $n$) such that $\nu^{x,k}$ is $(\varepsilon,m)$-atomic. However by Lemma \ref{MoranLower}, since $F_2$ is a Moran construction of positive lower dimension with strong separation condition and uniform branching number, $\nu$ cannot have any $(\varepsilon,m)$-atomic subtrees which is a contradiction.
\end{proof}

\subsection{Proof of Corollary \ref{main1cor}} \label{main1corproof}

Weak tangents were first introduced by Mackay and Tyson \cite{MT} and play a key role in calculating the Assouad dimension. Let $\mathcal{K}(\mathbb{R}^d)$ be the set of non-empty compact subsets of $\mathbb{R}^d$ equipped with the Hausdorff metric $d_{\mathcal{H}}$ defined by
\[
d_\mathcal{H}(A,B) = \inf \left\{ \varepsilon \ge 0 \colon A \subseteq [B]_\varepsilon \text{ and } B \subseteq [A]_\varepsilon \right\}
\]
where $[A]_\varepsilon$ is the closed $\varepsilon$-neighbourhood of a non-empty set $A$.

\begin{defn}
Let $X, E$ be  compact subsets of $\mathbb{R}^d$ with $E \subseteq X$ and $F$ be a closed subset of $\mathbb{R}^d$. Suppose there exists a sequence of similarity maps $T_k \colon \mathbb{R}^d \rightarrow \mathbb{R}^d$ such that $T_k(F) \cap X \rightarrow E$ in the Hausdorff metric. Then the set $E$ is called a \emph{weak tangent} to $F$.
\end{defn}

For simplicity and without loss of generality we will assume $X=[0,1]^d$ for the rest of this paper unless stated otherwise. The importance of weak tangents can be seen in the following propositions.
\begin{prop}\cite[Proposition 6.1.5]{MT}\label{mt}
Let $E,F\subseteq \mathbb{R}^d$, $E$ compact, $F$ closed and suppose $E$ is a weak tangent to $F$. Then $\Assouad F \ge \Assouad E$. 
\end{prop}
\begin{lma}\cite[Propositions 5.7-5.8]{Fr,KOR}\label{lma2}
	Let $F\subset \mathbb{R}^d$ be any nonempty closed set.  Then there is a weak tangent $E$ to $F$ such that $\Haus E=\Assouad F$.
\end{lma}

Lemma \ref{lma2} follows originally from Furstenberg's work in \cite{Fu1}, see also  \cite{Fu2}. This work was translated to our setting in \cite[Propositions 5.7-5.8]{KOR} and \cite{Fr}.  Applying weak tangents to sumsets we have the following lemma.

\begin{lma}\label{lma2.1}
	Let $F\subset \mathbb{R}^d$ be any nonempty closed set. Then for any weak tangent $E$ to $F$, $2E$ is a subset of a weak tangent to $2F$.
\end{lma}
\begin{proof}
This proof of this lemma is similar to the proof of \cite[Lemma 3.1]{Fr} but we include it for completeness. Assume $E$ is a weak tangent to $F$. This means that there is a sequence of similar copies $F_i$ of $F$ (under similarities $T_i$) such that $\lim_{i\to\infty}d_\mathcal{H}(F_i\cap [0,1]^d, E)=0$. It follows that
\[
\lim_{i\to\infty}d_\mathcal{H}(2(F_i\cap [0,1]^d), 2E)=0.
\]
We also note that 
\[
2(F_i\cap [0,1]^d) \subseteq 2(F_i) \cap [0,2]^d = (2F)_i \cap[0,2]^d,
\]
where $(2F)_i$ is the similar copy of $2F$ under $T_i$. As $(\mathcal{K}([0,2]^d),d_\mathcal{H})$ (the space of non-empty compact subsets of $[0,2]^d $ under the Hausdorff metric) is compact, there exists a weak tangent $G$ to $2F$ under the similarities $T_i$. Thus we have the following
\[
2E \leftarrow 2 (F_i \cap[0,1]^d) \subset (2F)_i \cap [0,2]^d \rightarrow G
\]
and, again as $(\mathcal{K}([0,2]^d),d_\mathcal{H})$ is compact, $2E \subseteq G$ as desired.
\end{proof}

We are now ready to complete the proof of Corollary \ref{main1cor}.

\begin{proof}	
	Let $F\subset\mathbb{R}$ be such that $0< \Assouad F < 1$.  Then by Lemma \ref{lma2} there is a weak tangent $E$ to $\overline{F}$ (the closure of $F$) with $\Haus E=\Assouad \overline{F} = \Assouad F$ (since Assouad dimension is stable under taking closure).  Therefore,
	\[
	0< \Haus E=\Boxd E=\Assouad E=\Assouad F < 1.
	\]
    By Lemma \ref{lma2.1} and Proposition \ref{mt} we see that 
    \[
   \Assouad 2F =   \Assouad 2\overline{F}\geq \Assouad 2E\geq\uBoxd 2E.
    \]
    Finally as $0 < \uBoxd E < 1$, we can apply Lemma \ref{lma1} to get
    \[
    \uBoxd 2E > \Boxd E = \Assouad F
    \]
as required.
\end{proof}

\subsection{Proof of Theorem \ref{main2}: convergence to 1}\label{main2proof}

\begin{proof}
	We can clearly assume $F$ is bounded and, as before, we can further assume $F$ is compact, since taking the closure does not effect the lower dimension.  Let $\Lower F=s>0$ then by our discussion in Section \ref{main1proof}, we can assume that $F$ is a Moran construction with strong separation condition and uniform branching number. Let $\nu$ be the probability measure on $F$ such that the measure of any dyadic interval $D$ intersecting $F$ is split equally between the next level dyadic intervals contained in $D$ and intersecting $F$ (so the measure defined in the proof of Lemma \ref{MoranLower}). As $\Lower F > 0$, we can find $\varepsilon>0$ and $m>0$ such that $\nu^{x,j}$ is never $(\varepsilon,m)$-atomic for every integer $j$ and $x\in F$. We note that $\varepsilon$ can be chosen arbitrarily small.

	Now let $\mu$ be any measure on $[0,1]$. Suppose that $\Lower \supp(\mu*\nu)=s'$ and by definition of the lower dimension, for any small $\gamma>0$, we can find dyadic intervals $E_i\in\mathcal{D}_{n_i}$ with a sequence $\{n_i\}_{i\in\mathbb{N}}$  and a sequence $m_i\to\infty$ such that $\mu*\nu(E_i)>0$ and
\begin{equation}\label{gamma}
	N(\supp(\mu*\nu)^{E_i},2^{-m_i})=N(\supp(\mu*\nu)\cap E_i, 2^{-(n_i+m_i)})\leq 2^{(s'+\gamma)m_i}.
\end{equation}
	As $\mu*\nu(E_i)>0$ we can find dyadic intervals $F_{1,i}, F_{2,i}\in\mathcal{D}_{n_i+1}$ such that $\mu(F_{1,i})>0, \nu(F_{2,i})>0$ and $F_{1,i}+F_{2,i}\subset E_i.$ Otherwise, by definition of the convolution, $\mu*\nu(E_i)=0$. Similarly we see that
	\[
	\mu^{F_{1,i}}*\nu^{F_{2,i}}\ll (\mu*\nu)^{E_i}.
	\]
	We denote $\mu_i=\mu^{F_{1,i}}$ and $\nu_i=\nu^{F_{2,i}}$. Now we estimate the entropy $H_{m_i}(\mu_i*\nu_i)$. We can apply Theorem \ref{HochmanEntropy} with $\varepsilon,m$ and obtain constants $\delta=\delta(\varepsilon,m), n_0=n_0(\varepsilon,m)$. As $\nu^{x,j}$ is never $(\varepsilon,m)$-atomic, the same holds for $\nu^{x,j}_i$. Thus we see that for any $n>n_0$ there exists a subset $I_n\subset\{1,\dots,n\}$ with cardinality at least $(1-\varepsilon)n$ such that either
	\[
	H_n(\mu_i*\nu_i)\geq H_n(\mu_i)+\delta
	\]
	or
	\[
	\mu_i (\{ x \in [0,1] :  \mu_i^{x,k} \text{ is } (\varepsilon,m)\textup{-uniform}\})>1-\varepsilon, \text{ if }k\in I_n.
	\]
	In the latter case we see from the proof of Lemma \ref{EntropyToCovering} that
	\[
	H_n(\mu_i,\mathcal{D}_n)\geq (1-\varepsilon)^3n.
	\]
	This in turn implies that there exists a constant $C$ depending only on $m$ such that
	\[
	N(\supp(\mu_i*\nu_i),2^{-n})\geq N(\supp(\mu_i),2^{-n})\geq C 2^{(1-\varepsilon)^3n}.
	\]
	When the above holds at scale $n=m_i$, for all large enough $i$, we obtain the following 
	\[
	N(\supp(\mu*\nu)^{E_i},2^{-m_i})\geq C 2^{(1-\varepsilon)^3m_i}.
	\]
	Thus by equation \eqref{gamma} we see that $s'\geq (1-\varepsilon)^3-\gamma$.   Otherwise we are in the first case for infinitely many $i$ such that $n=m_i$. Then we have
	\[
	H_n(\mu_i*\nu_i)\geq H_n(\mu_i)+\delta,
	\]
	and so for such $m_i$
	\[
	N(\supp(\mu*\nu)^{E_i},2^{-m_i})\geq 2^{m_iH_{m_i} (\mu_i*\nu_i)}\geq 2^{m_i (H_{m_i}(\mu_i) + \delta )}.
	\]
	Again by equation \eqref{gamma}, this implies that for infinitely many $i$
	\[
	s'\geq H_{m_i}(\mu_i)+\delta-\gamma.
	\]	
    We have so far not made any assumptions about $\mu$. As the lower dimension of $F$ is positive, the lower dimension of $kF$ is also positive for any integer $k$. Thus we can consider a Moran construction subset of $kF$, denoted $G$ and define $\mu$ to be the measure on $G$ such that the measure of a dyadic interval is equally distributed among its next level descendants.
    
    Then since $G$ is a Moran construction as in Lemma \ref{LowerDimensionSet}, we see that $H_{m_i}(\mu_i)\geq \Ld G-\gamma$ when $i$ is large enough. Thus
	\[
	s'\geq \Ld G+\delta-2\gamma.
	\]
    
	Combining the two cases, as $\gamma>0$ can be arbitrarily chosen, we see that
	\[
	 \Lower (k+1)F \ge \Lower F+G = s' \geq \min\{\Ld G+\delta, (1-\varepsilon)^3\}.
	\]   
   As a result we see that
	\[
	\Ld (k+1)F \geq (1-\varepsilon)^3
	\]
	or
	\[
	\Ld (k+1)F \geq \Ld kF+\delta.
	\]
	Here we see that $\delta$ does not depend on $k$, therefore for all $k$ large enough
	\[
	\Ld (k+1)F \geq (1-\varepsilon)^3.
	\]
	But now we can choose $\varepsilon\to 0$ so we see that
	\[
	\lim_{n\to\infty}\Ld nF=1
	\]
as required.
\end{proof}

\subsection{Assouad dimension of distance sets}\label{distanceproofA}

We begin by proving a weaker version of Theorem \ref{distancethmA}, where one does not have the strict inequality.  This result is simpler to prove, although the method is philosophically similar and so this proof will shed light on the proof of the stronger result which follows.  

\begin{lma}\label{distancelma}
If $F\subseteq [0,1]^d$, then 
\[
\Assouad D(F) \ge \frac{1}{d}\Assouad F.
\]
\end{lma}
\begin{proof}
We first deal with the 2-dimensional case, and then our method will be generalised to higher dimensions. 

Let $F\subseteq [0,1]^2$, $s = \Assouad D(F)$ and  $\varepsilon>0$.  Let $x \in F$ and $0<r<R<1$.  We wish to construct an $r$-cover of $F \cap B(x,R)$ using the  distance set. The Assouad dimension tells us roughly how many intervals of length $r$ are needed to cover part of the distance set. If an interval, say $[a,a+r]$, is needed in the cover of $D(F)$ then there is a point $x\in F$ such that the annulus $\left\{ y \colon \lvert y-x \rvert \in [a,a+r] \right\}$ intersects $F$ at least once. For $x' \in \mathbb{R}^2$ and $a,\Delta \in [0,1]$ we define the annulus around $x'$ with width $\Delta$ and inner radius $a$ by
\[
S(x',a,\Delta) = \left\{ y\in \mathbb{R}^2 \colon \lvert y-x' \rvert \in [a,a+\Delta] \right\}.
\]
In fact we will only use annuli of the form $S(x',i\Delta, \Delta)$ for some $\Delta$ and $i=0,1,2,\ldots$.  We first ask, how many of the annuli of this form can intersect $F$. Let $I\subset \mathbb{N}$  be the set of  integers $i$ such that
\[
D(F) \cap [ir, (i+1)r] \neq \emptyset.
\]
It follows that
\[
\#\left( I \cap \left[ 0, \frac{R}{r} \right] \right) \le C \left( \frac{R}{r} \right)^{s+\varepsilon}
\]
where $C = C(\varepsilon)>0$ is the constant coming from the definition of the Assouad dimension of $D(F)$.   Suppose $i\in I$ is such that $F \cap S(x,ir,r) \neq \emptyset$ and $i \geq 10$.  Choose  $y \in F \cap S(x,ir,r)$ and consider annuli $S(y,jr,r)$ around $y$ for $j=0,1,2,\ldots$.  Observe that if $S(x,ir,r) \cap S(y,jr,r) \cap F \neq \emptyset$, then $j \in I$.  Moreover, if $jr < 1.9 ir$ then $S(x,ir,r) \cap S(y,jr,r)$ can be covered by a uniform constant $C'$ many balls of radius $r$.  It remains to cover $F \cap S(x,ir,r) \setminus B(y, 1.9ir)$.  If this is empty, then we are done, and if it is not empty then fix $z \in F \cap S(x,ir,r) \setminus B(y, 1.9ir)$ and cover the remaining portion as above using $z$ in place of $y$. It follows that
\[
N(S(x,ir,r) \cap F) \leq 2 C' \#\left( I \cap \left[ 0, \frac{R}{r} \right] \right)  \leq 2 C' C \left(\frac{R}{r}\right)^{s+\varepsilon}.
\]
Since $B(x,10r)$ can be covered by a constant $C''$ many $r$-balls, we conclude
\[
N(B(x,R) \cap F, r) \leq C'' +  2 C' C \left(\frac{R}{r}\right)^{s+\varepsilon} \times  \#\left( I \cap \left[ 10, \frac{R}{r} \right] \right)  \leq C''+ 2 C' C^2 \left(\frac{R}{r}\right)^{2s+2\varepsilon}
\]
which proves that $\Assouad F \le 2 s+2 \varepsilon$ and letting $\varepsilon \to 0$ yields $\Assouad D(F) \ge \Assouad F/2$ as required.

The $d$-dimensional case follows precisely from the above argument plus an observation we call `dimension reduction'.  The main idea above was to divide the plane into two collections of $r$-thin annuli so that the  intersection of two annuli (one from each collection) was essentially an $r$-ball. We do the same thing in the $d$-dimensional case, but this time the intersection of two annuli is essentially a  $(d-1)$-dimensional annulus which is also $r$-thin.   This dimension reduction strategy is iterated $(d-2)$-times  until we end up with 2-dimensional annuli and then our previous covering argument applies. We end up estimating
\[
N(B(x,R) \cap F, r) \leq  C(d) \#\left( I \cap \left[ 0, \frac{R}{r} \right] \right) ^d \leq C(d) C^d \left(\frac{R}{r}\right)^{ds+d\varepsilon},
\]
where $C(d)$ is a constant depending on the ambient spatial dimension.  This proves the desired result. 
\end{proof}

Adapting this proof to obtain the strict inequality in Theorem \ref{distancethmA} is non-trivial but follows the same idea with an additional application of the inverse theorem.

\begin{proof}[Proof of Theorem \ref{distancethmA}]
Again we start with the planar case and assume $\Assouad D(F)=s \in (0,1)$, noting that if $\Assouad D(F)=1$, the result is trivial.  Let $\varepsilon \in (0,1/2)$ and fix $x \in F$ and $0<r<R<1$. Follow the argument and notation above exactly, until it comes to covering $S(x,ir,r)$. Here, instead of decomposing this annulus into balls of radius $r$ we use relatively long and thin rectangles and then cover each rectangle separately. 

First we cover $S(x,ir,r)$ by an optimal number of equally spaced $2r$ by $r\sqrt{2i-1} $ rectangles as illustrated in Figure \ref{fig:deco}.
\begin{figure}[h]
	\caption{Decomposing an annulus into rectangles}
	\centering
	\begin{tikzpicture}[scale=0.5, every node/.style={scale=0.5}]
\node [draw, thick, circle, minimum width=400pt] {};
\node [draw, thick, circle, minimum width=355pt] {};
\node [draw, thick, circle, minimum width=10pt, fill] {};
\draw (0,0) -- (7,0);
\node[label=below:{\LARGE $ir$}] at (2,0) {};
\node[label=below:{\LARGE $2r$}] at (5.6,-3) {};
\node[label=left:{\LARGE $r\sqrt{2i-1}$}] at (5.3,1) {};
\draw [very thick] (1.743,6.01812) -- (6.04187,1.71937);

\transparent{0.5}\draw [fill=black!30, very thick] (-5.4,-3) rectangle (-7,3);
\draw [fill=black!30, very thick] (5.4,-3) rectangle (7,3);
\draw [fill=black!30, very thick] (1.743,6.01812) -- (2.8284,7.07) -- (7.07106,2.82842) -- (6.04187,1.71937);
\end{tikzpicture}
	\label{fig:deco}
\end{figure}
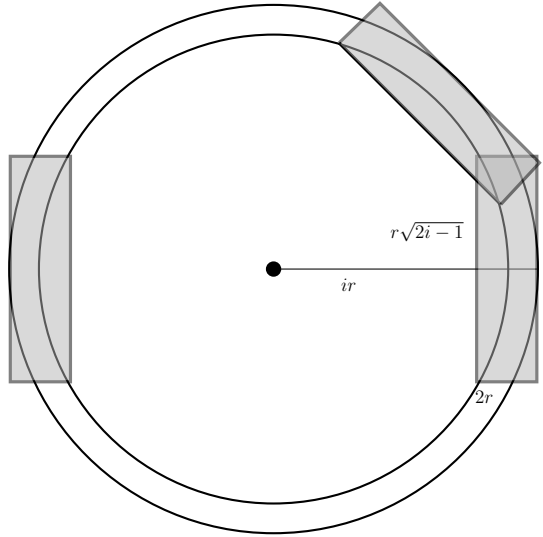
Suppose $i\in I$ is such that $F \cap S(x,ir,r) \neq \emptyset$ and $i \geq 10$.  Choose  $y \in F \cap S(x,ir,r)$ and consider distances from $y$ to points in $S(x,ir,r)$ as above.  It follows that there is an absolute constant $A$ such that at most
\[
A \left(\frac{R}{\sqrt{i}r} \right)^{s+\varepsilon}
\]
of the previously defined rectangles covering $S(x,ir,r)$ can intersect $F \cap S(x,ir,r)$.  We will cover the part of $F$ lying inside each of these rectangles separately using the natural partition of the rectangle into squares of sidelength $2r$ oriented with the rectangle.   Fix a rectangle and denote the  associated collection of $2r$-squares which optimally cover the part of $F$ inside this rectangle by $\mathcal{S}$.  Also let $D= D(F\cap S) \subseteq D(F)\cap [0,r\sqrt{2i-1}]$.

For each $S\in \mathcal{S}$ we write  $x_S$ to denote the centre of the square $S$ and let $X$ be the set of all $x_S$. Then 
\[
B(x_{S_1}-x_{S_2}, 4 \times 2r) \cap (F-F) \neq \emptyset
\]
for all  $S_1,S_2 \in S$. Therefore there is a point $y\in D$ such that $\lvert y - \lvert x_{S_1}-x_{S_2} \rvert \rvert \le 4\times 2r$. From this fact we see that the difference set $X-X$ and the set of distances $D$ are closely related in that
\[
N(D, 4\times 2r) \le N(X-X, 2r) \le 4 N(D, 4\times 2r).
\]
All the points in $X$ lie on the same straight line segment and therefore we can consider them as a subset of the unit interval and thus  use Hochman's inverse theorem. The tree $T_{D}$  associated to $D$  is  a subtree of $T_{D(F)}$ and by our assumption that $D(F)$ does not have full Assouad dimension, there exists $\varepsilon_1>0, m_0>0$ such that $T_{D(F)}$ (and therefore $T_D$) does not have any $(\varepsilon_1,m)$-branching subtrees with $m$ greater than or equal to $m_0$. We can choose $\varepsilon_1$ to be arbitrarily small.

Recall the inverse theorem. For any pair of numbers $m,\varepsilon_1$, there is a $\delta(\varepsilon_1,m)>0$ and a $\rho_0\in (0,1)$ such that whenever $\rho<\rho_0 $, for any finite set $K \subseteq [0,1]$, by formula (\ref{treecountingbd}) from the proof of Lemma \ref{lma1} either $N(K-K,\rho) > N(K,\rho)^{1+\delta}$ or $T_K$ contains $(\varepsilon_1,m)$-full branching subtrees or
    \[
	N(K,\rho)\leq 2^{-5\varepsilon_1\log \rho}.
	\]
Now we can properly  choose  our parameters $\varepsilon_1$ and $m$. Since $s>0$ we can choose $\varepsilon_1$ such that
\[
5\varepsilon_1  <\frac{s+\varepsilon}{\delta + 1}\tag{*}
\]
and $m=[\log_2 1/\varepsilon_1] > 2m_0+1$, where $\delta$ is the $\delta(\varepsilon_1,m)$ from the inverse theorem. When choosing our $\varepsilon_1$, the $\delta$ will shrink as $\varepsilon_1$ does so there always exists an $\varepsilon_1$ satisfying the inequality. From now on $\varepsilon_1$ and $m$ shall be considered as constants. As a consequence, $\delta$ and $\rho_0$ can be considered as constants as well.

In the following, we shall assume that $i$ is large enough so that $\frac{1}{\sqrt{2i-1}}<\rho_0$. This will not cause any loss of generality (for example we can replace the condition $i\geq 10$ by $i\geq \rho^{-10}_0$).

The tree $T_X$ associated with $X$ cannot have any full branching subtrees of height $m$ as this would imply there exists a full branching subtree of height at least $m_0$ in $T_{D}$ which contradicts the assumption that $T_{D(F)}$ does not have $(\varepsilon_1,m)$-full branching subtrees with $m$ greater than or equal to $m_0$. We  scale our set $X$ by $(r\sqrt{2i-1})^{-1}$ to obtain a set $X' \subset  [0,1]$, noting that such rescaling will not change the tree structure and therefore applying the inverse theorem to $X'$ as we did with $K$ above, with  $\rho=2/\sqrt{2i-1}$, we see that either 
\[
N\left(X',\frac{2}{\sqrt{2i-1}}\right)^{1+\delta}<N\left(X'-X',\frac{2}{\sqrt{2i-1}}\right)
\]  
or
\[
N\left(X',\frac{2}{\sqrt{2i-1}}\right)\leq \left(\frac{2}{\sqrt{2i-1}}\right)^{-3\varepsilon_1}.
\]
Scaling covers back to the original set $X$,  we see that either
\[
N(X, 2r)^{1+\delta} < N(X-X, 2r) \le 4N(D, 4\times 2r) \le 4 C_\varepsilon \left(\frac{r\sqrt{2i-1}}{4\times 2r}\right)^{s+\varepsilon}=4C_\varepsilon \left( \frac{\sqrt{2i-1}}{8}\right)^{s+\varepsilon}
\]
or
\[
N(X,2r) \le \left(\frac{2}{\sqrt{2i-1}}\right)^{-5\varepsilon_1}.
\]
Recalling  (*), this guarantees that there is a constant $A'$ such that, for each rectangle $S\in\mathcal{S}$, we have
\[
N(S\cap F , 2r) \le A' \sqrt{i}^{\frac{s+\varepsilon}{1+\delta}}.
\]
This holds for all $i\ge \max\left\{C_{\varepsilon}^2, \frac{1}{\rho_0^{10}},10\right\}=:i_0$. For smaller values of $i$ we only need a constant $C(\varepsilon,\rho_0)$ of balls to cover the rectangles. In conclusion
\[
N(B(x,R)\cap F,r) \le 4A'A \sum_{i\in I \cap \left[i_0,R/r\right]} \sqrt{i}^{\frac{s+\varepsilon}{1+\delta}} \left(\frac{R}{r\sqrt{i}} \right)^{s+\varepsilon} + C(\varepsilon,\rho_0).
\]
We bound this sum using the following simple general inequality.  Let  $Z \subset \mathbb{Z}^+$ be a finite set of positive integers and $t \in (0,1)$.  Then
\[
\sum_{i \in Z} i^{-t} \leq \sum_{i =1}^{\#Z} i^{-t} \leq \int_0^{\#Z} x^{-t} dx = \frac{1}{1-t}(\#Z)^{-t+1}.
\]
Applying this inequality in our setting, where we have $t= \frac{s+\varepsilon}{2}(\frac{1}{1+\delta}-1) \in (0,3/4)$, yields
\[
\sum_{i\in I \cap [i_0,R/r]}i^{\frac{s+\varepsilon}{2}(\frac{1}{1+\delta}-1)} \le 4\left(\# I \cap \left[i_0,R/r\right]\right)^{\frac{s+\varepsilon}{2}(\frac{1}{1+\delta}-1)+1}
\]
and therefore
\[
N(B(x,R)\cap F, r) \le A''\left( \frac{R}{r} \right)^{(s+\varepsilon)(\frac{s+\varepsilon}{2}(\frac{1}{1+\delta}-1)+2)}+C(\varepsilon,\rho_0)
\]
for a uniform constant $A''$.  This proves that 
\[
\Assouad F \leq (s+\varepsilon)\left(\frac{s+\varepsilon}{2}\left(\frac{1}{1+\delta}-1\right)+2\right)
\]
and letting $\varepsilon \to 0$ yields
\[
\Assouad F \leq 2s - \frac{s^2\delta }{2(1+\delta)} < 2s
\]
as required.

For sets in $\mathbb{R}^d$ we use the dimension reduction technique introduced in the previous lemma and then use the rectangles from this proof instead of picking two points in an annulus. This gives us 
\[
\Assouad F \le (d-2)(s+\varepsilon)+ \left(s+\varepsilon\right)\left(\frac{s+\varepsilon}{2}\left(\frac{1}{1+\delta}-1\right)+2\right)
\]
and the right hand side is strictly less than $d s $ for small enough $\varepsilon$, concluding the proof.
\end{proof}

\subsection{Box dimension of distance sets}\label{distanceproofB}

In this section we show that a similar distance set result holds for the upper box dimension. Unlike the Assouad dimension, which is `local', the box dimensions are `global'. This prevents the distance set cutting method introduced in the previous section from working. Instead, we use the pigeonhole principle iteratively to reduce the dimension down to the 1-dimensional case and then we can apply the inverse theorem.

\begin{proof}[Proof of Theorem \ref{distancethmB}]
Let $r=2^{-n}$ for some integer $n>0$. Let $C_F(r)$ and $C_{D(F)}(r)$ be the collections of cubes in the standard $r$-meshes which intersect $F$ and $D(F)$, respectively, and write $N(F,r)$ and $N(D(F),r)$ as the cardinalities of $C_F(r)$ and $C_{D(F)}(r)$, respectively.

There are $N(F,r)^2$ pairs of cubes in $C_F(r)$ and for each pair $(i,j)$, $i,j \in C_F(r)$, the set of distances between the points of $F$ in one cube and the points in the second, denoted as $D(i,j)$, is contained in an interval of length $c_d r$ where $c_d$ is a constant depending only on $d$. Clearly $D(i,j) \subset D(F)$. 

For each cube $K\in C_{D(F)}(r)$, let
\[
n_K= \# \left\{ (i,j) \in C_F(r) \times C_F(r) \colon D(i,j) \cap K \neq \emptyset \right\}.
\]
We have the following inequality
\[
\sum_{K\in C_{D(F)}(r)} n_K \ge N(F,r)^2
\]
and there must exist at least one $K_0\in C_{D(F)}(r)$ such that $n_{K_0}\geq \frac{N(F,r)^2}{N(D(F),r)}$. By the pigeonhole principle there exists at least one $i\in C_F(r)$ such that
\[
\#  \left\{j\in C_{F}(r) \colon D(i,j) \cap K_0 \neq \emptyset \right\} \ge \frac{n_{K_0}}{N(F,r)}.
\]
In other words, there exists an $x\in F$ and $y \in D(F)$ such that the annulus $S(x,y,c_d r)$ intersects at least $\frac{n_{K_0}}{N(F,r)}$ many cubes in $C_F(r)$. We assume $y$ is `large' compared to $r$, say $y>Mr$, for otherwise the number of cubes intersected by the annulus is bounded above by a constant
\[
\frac{n_{K_0}}{N(F,r)} \le M^d.
\]
Here $M$ is a constant which will be specified later.

We wish to further decompose this annulus. An easy first step is to split it into $2^d$ quadrants, that is, we perform a change of basis so that $x$ is the origin and regroup elements of the annulus whose coordinates all have the same signs, so $\alpha=(\alpha_1,\ldots,\alpha_d)$ and $\beta=(\beta_1,\ldots,\beta_d)$ are in the same quadrants  if $\text{sign } \alpha_i = \text{sign } \beta_i$ for all $i=1,\ldots,d$. Again by the pigeonhole principle at least one of these quadrants will intersect at least 
\[
N_1= 2^{-d} \frac{n_{K_0}}{N(F,r)}
\]
many cubes from $C_F(r)$. This reduction will ensure a certain transversality condition holds below.
    
Now we  iterate the above argument. In the chosen quadrant there are $N_1^2$ many pairs of cubes that intersect $S(x,y,c_d r)$ and $F$. The distances between points in these cubes are all contained within a $c_dr$-interval, and by the same pigeon hole strategy as above we find a point $x_2\in F \cap S(x,y,c_d r)$ and a $y_2\in D(F)$ such that $S(x_2,y_2,c_dr)$ intersects at least
	\[
	\frac{N^2_1}{N_1 N(D(F),r)}
	\]
	many cubes which are in $C_F(r)$ and at the same time intersect $S(x,y,c_d r)$. The intersection of two specific $d$-dimensional $r$-thin annuli is  contained in a $c'_d r$-neighbourhood of a $(d-2)$-sphere, for some constant $c_d'$ depending only on $d$. Decompose the sphere into $2^{d-1}$ `quadrants'  as before (where we think of the centre of the sphere as the origin), and we can find a quadrant  intersecting at least
	\[
	N_2=2^{-d+1}\frac{N^2_1}{N_1 N(D(F),r)}
	\]
	many cubes in $C_F(r)$.

We can perform the above `dimension reduction' argument  $(d-1)$ times to end up with (a piece of) a $1$-sphere whose $c'_dr$ neighbourhood intersects at least
	\[
	c''_d \frac{n_{K_0}}{N(D(F),r)^{d-2}N(F,r)}
	\]
	many cubes in $C_F(r)$. Here $c''_d$ is another constant depending on $d$.   Also if for some $m$, we have $y_m<Mr$, then 
	\[
	\frac{n_{K_0}}{N(D(F),r)^{m}}\leq M^d.
\]
In the case that each $y_m$ is larger than $Mr$, we end up with a piece of a $1$-sphere whose $c'_d r$-neighbourhood (which is just an annulus) contains a large number of cubes in $C_F(r)$. Our first observation is that there exists an absolute constant $a_d>0$ such that for all $r$ small enough we have the following inequality
\[
\frac{n_{K_0}}{N(D(F),r)^{d-2}N(F,r)}\leq a_d N(D(F),r)\tag{$\dagger$}.
\]
To see this, recall that in the last of the above iterations we found disjoint $r$ cubes. Those cubes are contained in a neighbourhood of radius $c_d r$ of a (piece of a) $1$-sphere. We  enumerate these cubes by $\{C_1,C_2,\dots, C_Z\}$ for a  suitable integer $Z$, and choose $x_i\in F\cap C_i$ for all $i\in\{1,\dots,Z\}$. Consider the following set
\[
X=\{|x_1-x_2|,\dots,|x_1-x_Z|\}.
\] 
It is not hard to show that there exists an absolute constant $v_d>0$ such that $X$ is a $v_dr$-separated set. Also it is clear that $X\subset D(F)$. From here we see that inequality  ($\dagger$) follows. Then we see that, by the choice of $K_0$, the following inequalities hold
\[
c''_d \frac{N(F,r)}{N(D(F),r)^{d-1}}\leq c''_d\frac{n_{K_0}}{N(D(F),r)^{d-2}N(F,r)}\leq a_d N(D(F),r).
\]
This implies that for all $r$ small enough
\[
N(F,r)\leq \frac{a_d}{c''_d} N(D(F),r)^d,
\]
and therefore
\[
\uBoxd F\leq d \, \uBoxd D(F).
\]
This concludes the first part of this theorem. To see the second part we shall use the circle decomposition as well as the inverse theorem as in the proof of Theorem \ref{distancethmA}. We want to make use of the arithmetic structure of the set $Y=\{x_1,\dots, x_Z\}.$ However, $Y$ has `curvature' and so we cannot directly apply the inverse entropy theorem for $Y-Y$. As in the proof of Theorem \ref{distancethmA} we first decompose  $Y$ into almost straight pieces and use the inverse entropy theorem for each straight piece. Then we see that for all small enough $r>0$, if $y\in D(F)$ and $y>cr$ then the covering number $N(Y,r)$  can be bounded from above by
	\[
	C N\Big(D(F)\cap[0,\sqrt{yr}],r\Big)^{\frac{1}{1+\delta}}N(D(F)\cap [0,2\pi y],\sqrt{yr}),
	\] 
where $c, \delta,C>0$ are constants that depend on $F$.   We  now fix  $M=c$ above. We see that for a constant $c'''_d>0$
\[
	c'''_d \frac{n_{K_0}}{N(D(F),r)^{d-2}N(F,r)}\leq N(Y,r).
\]
Let $\varepsilon>0$.  Appealing directly to the box dimension and Assouad dimension of $D(F)$, we can find an absolute constant $C' = C'(\varepsilon)>0$ such that
	\[
	N(D(F)\cap[0,\sqrt{yr}],r)\leq C'(\sqrt{y/r})^{\Assouad D(F)+\varepsilon},
	\]
	\[
	N(D(F)\cap [0,2\pi y],\sqrt{yr})\leq  C'(\sqrt{y/r})^{\Assouad D(F)+\varepsilon}
	\]
   and
          \[
N(D(F),r)\leq  C'(1/r)^{\uBoxd D(F)+\varepsilon}.
\]
Combining these estimates with the inequality established above yields
	\[
	c'''_d \frac{N(F,r)}{N(D(F),r)^{d-1}}\leq c'''_d\frac{n_{K_0}}{N(D(F),r)^{d-2}N(F,r)}\leq CC'^2 (\sqrt{y/r})^{(\Assouad D(F)+\varepsilon)/(1+\delta)}(\sqrt{y/r})^{\Assouad D(F)+\varepsilon}
	\]
and this implies that
\begin{eqnarray*}
N(F,r)&\leq& C{C'}^{d+1}{c'''_d}^{-1} (1/r)^{(d-1)(\uBoxd D(F)+\varepsilon)} (\sqrt{y/r})^{(\Assouad D(F)+\varepsilon)/(1+\delta)}(\sqrt{y/r})^{\Assouad D(F)+\varepsilon}\\
&\le &C{C'}^{d+1}{c'''_d}^{-1}\left(\frac{1}{r}\right)^{(d-1)\uBoxd D(F)+ (1+1/(1+\delta))\Assouad D(F)/2+d\varepsilon}.
\end{eqnarray*}
Therefore we see that
    \[
    \uBoxd F \leq \left( d-1\right) \uBoxd D(F)+\frac{1+(1+\delta)^{-1}}{2} \Assouad D(F) < \left( d-1\right) \uBoxd D(F)+\Assouad D(F)
    \]
as required.	
\end{proof}

\section{Further comments and examples}

As we proved in Corollary \ref{main1sym}, if a set $F\subset\mathbb{R}$ satisfies
\[
\uBoxd 2F=\uBoxd F,
\]
then  either $\uBoxd F=0$ or $\Assouad F=1$. A partial converse also holds trivially. If $\uBoxd F=0$ then 
\[
\uBoxd 2F=\uBoxd F=0 \text{ and } \uBoxd D(F)=\uBoxd F=0.
\]
For Assouad dimension, the situation is rather different.  Concerning distance sets, \cite[Example 2.6]{Fr} provides  an example of a $F\subset [0,1]$ with $\Assouad F=0$ and $\Assouad D(F)=1$ and we can easily use this example to build similar examples for sumsets. Let $F_1=F\cup (-F)$, and observe that $\Assouad F_1=0$ and $2F_1\supset D(F)$, and so $\Assouad 2F_1=1$. 

Positive lower dimension is not a necessary condition for the box dimensions of the iterated sum sets to approach $1$.  We demonstrate this by considering a  simple example where  $F =\{1/k\}_{k\in\mathbb{N}}$.  Clearly, the lower (and modified lower) dimension of $nF$ is 0 for all $n$, but we can show that $\lBoxd n F \to 1$ (even at an exponential rate).  
\begin{prop}
For $F =\{1/k\}_{k \in\mathbb{N}}$ and $n \geq 1$, we have
\[
\lBoxd n F\geq 1-2^{-n}.
\]
\end{prop}

\begin{proof}
Given $\delta>0$ we say a set $E$ is $\delta$-dense in a closed  interval $I$ if every point in $I$ is at distance less than $\delta$ from some  point in $E$.  Suppose $E$ is $\delta$-dense in $[0,t]$ for some small $t \in (0,1)$.  Choose $k \in \mathbb{N}$ such that $1/k < \sqrt{t} \leq 1/(k-1)$.  It follows that $\sqrt{t}-1/k \leq t$ and so $E+F$ must be $\delta$-dense in $[0,\sqrt{t}]$.  Since $F$ is easily seen to be $\delta$-dense in $[0,\sqrt{\delta}]$ it follows by induction that $nF$ is $\delta$-dense in $[0, \delta^{2^{-n}}]$.  Therefore
\[
N(nF, \delta) \  \geq \ \delta^{2^{-n}} / \delta 
\]
and so $\lBoxd n F\geq 1-2^{-n}$ as required.
\end{proof}

\subsection{Self-similar sets} \label{selfsimsection}

If one considers restricted families of sets, then often more precise information can be obtained concerning the sumsets.  A particular setting which has received a lot of attention is that of  \emph{self-similar sets}, see \cite[Chapter 9]{Fa} for basic definitions and background on iterated function systems (IFSs).  In \cite{PS} it was shown that if $F \subseteq [0,1]$ is a self-similar set where two of the defining  contraction ratios $r_1,r_2$ satisfy $\frac{\log r_i}{\log r_j}\notin\mathbb{Q}$ then
\[
\Haus 2F=\min \{1,2\Haus F\}.
\]
Takahashi \cite{Ta} proved that if the sum of the dimensions of two self-similar Cantor sets  exceeds 1, then one can find new Cantor sets, arbitrarily close to the original ones, such that there is an interval in the sumset. Other related papers where the problem of finding an interval in the sumset or iterated sumsets of Cantor sets  include \cite{astels, hare1, hare2}.

We provide a simple argument demonstrating that the  dimensions of the iterated sumsets of a self-similar set reach 1 in \emph{finite time}. 

\begin{prop}\label{IFSthm}
Let $F \subseteq \mathbb{R}$ be a self-similar set which is not a singleton.  Then for some $n \geq 1$, the iterated sumset $nF$ contains an interval and therefore has Hausdorff, box and Assouad dimensions equal to 1.
\end{prop}

This result obviously extends to sets containing non-singleton self-similar sets, which include (non-singleton) graph-directed self-similar sets, subsets of self-similar sets generated by irreducible subshifts of finite type, and many examples of $\times p$ invariant subsets of $S^1$.

\begin{proof}[Proof of Proposition \ref{IFSthm}]

Suppose $F \subseteq [0,1]$ is a self-similar set which is not a singleton.  Then it necessarily contains a self-similar set which is generated by an IFS consisting of two orientation preserving maps with the same contraction ratio and which satisfies the strong separation condition.  To see this, choose two maps with distinct fixed points and iterate each an even number of times until the images of some large interval under the two iterated maps are disjoint.  Composing these two maps with each other in the two possible orders yields an IFS with the desired properties.  We may also renormalise so that the maps fix 0 and 1 respectively.  Since sumsets are monotone in the sense that $E \subseteq F \Rightarrow nE \subseteq nF$ for all $n$, it suffices to prove the result  for self-similar sets generated by IFSs $\Phi = \{\phi_1, \phi_2\}$ where $\phi_1, \phi_2: [0,1] \to [0,1]$ are defined by $ \phi_1(x)= r x$ and $\phi_2(x)=rx + (1-r)$ where $r \in (0,1/2)$ is a common contraction ratio. We write $X(\Phi)$ for the attractor of $\Phi$ and $k \Phi$ to denote the IFS with  common contraction ratio $r$  but with  translations taking all values in the iterated sumset $kT$ where $T = \{0, 1-r\}$ is the set of translations associated with $\Phi$.  We also write $X(k\Phi)$ for the attractor of this IFS and observe that for any integer $k$, $k X(\Phi) = X(k \Phi)$.

Note that $\text{diam}(X(k\Phi))= k$ and so
\[
\text{diam}(\phi_i(X(k\Phi)))=r k
\]
where $\phi_i$ is any map in $k\Phi$.  The set of translations defining $k\Phi$ is 
\[
k \left\{0,(1-r) \right\}= \left\{ n(1-r) \colon n=0,1,\ldots, k\right\}
\]
and therefore for all distinct $\phi_i, \phi_j \in k\Phi$ we see that $\lvert\phi_i(0)-\phi_j(0)\rvert \geq 1-r$ (independent of $k$). Thus there are $k+1$ maps in  $k\Phi$, and the IFS satisfies the strong separation condition as long as $rk  < 1-r$.  However,  for  $k\geq (1-r)/r$, the interval $[0,k]$ is invariant under $k\Phi$ which implies $X(k\Phi) =[0,k]$  completing the proof.
\end{proof}

\section{Acknowledgement}

Some of this work was completed while the authors were  resident at the  Institut Mittag-Leffler during the semester programme \emph{Fractal Geometry and Dynamics} and they are grateful for the inspiring atmosphere and financial support.   The authors thank Xiong Jin, Tuomas Sahlsten,  Pablo Shmerkin, Meng Wu,  and Josh Zahl  for helpful remarks and also the participants of the 2017 St Andrews reading group on additive combinatorics which stimulated some of this work.  They also thank an anonymous referee for carefully reading the paper and making several helpful suggestions.

\providecommand{\bysame}{\leavevmode\hbox to3em{\hrulefill}\thinspace}
\providecommand{\MR}{\relax\ifhmode\unskip\space\fi MR }
\providecommand{\MRhref}[2]{%
	\href{http://www.ams.org/mathscinet-getitem?mr=#1}{#2}
}
\providecommand{\href}[2]{#2}

\end{document}